\documentclass[10pt]{amsart}

\usepackage{ifpdf}

\usepackage{amsmath}
\usepackage{amsthm}
\usepackage{amsfonts}
\usepackage{amssymb}
\usepackage{subfig}
\usepackage{tikz,pgfplots}
\usepgfplotslibrary{fillbetween}
\usetikzlibrary{intersections}
\usetikzlibrary{patterns, snakes}
\usepackage[colorlinks,citecolor=blue,linkcolor=blue]{hyperref}

\usepackage{amscd}	
\usepackage{microtype} 
\usepackage{multirow}
\usepackage{color}

\ifdefined\draftversion
    \providecommand{\showkeys}{}
    \usepackage[shortalphabetic,initials,backrefs]{amsrefs}
\else
    \usepackage[initials]{amsrefs}
\fi

\ifdefined\showkeys
	\usepackage[color]{showkeys} 
\fi

\makeatletter
\def\blfootnote{\xdef\@thefnmark{}\@footnotetext}
\makeatother

\IfFileExists{.git/git.tex}{
\input{.git/git.tex}

}
{

\newcommand{\gitshash}{NA}

\newcommand{\gitauthsdate}{NA}
}

  \newcommand{\fig}[2]{
    \IfFileExists{#1.pdf_tex}{
      \def\svgwidth{#2}\input{#1.pdf_tex}
    }{
      \frame{Missing figure ``#1.pdf\_tex''}
      \message{LaTeX Warning: Missing figure ``#1.pdf\_tex'' on input line \the\inputlineno}
    }
  }

\newcommand{\dimension}{N}

\newcommand{\interior}{\operatorname{int}}

\newcommand{\dx}{\;dx}

\newcommand{\diff}[1]{\;d{#1}}

\newcommand{\pth}[1]{\left(#1\right)}

\newcommand{\set}[1]{{\left\{#1\right\}}}

\newcommand{\ang}[1]{{\left\langle#1\right\rangle}}
\newcommand{\norm}[1]{\left\|#1\right\|}

\newcommand{\cl}[1]{\overline{#1}}	

\newcommand{\e}{\ensuremath{\varepsilon}}

\newcommand{\Q}{\ensuremath{\mathbb{Q}}}
\newcommand{\R}{\ensuremath{\mathbb{R}}}

\newcommand{\Rd}{\ensuremath{{\mathbb{R}^{\dimension}}}}
\newcommand{\Rn}{\Rd}
\newcommand{\Z}{\ensuremath{\mathbb{Z}}}
\newcommand{\N}{\ensuremath{\mathbb{N}}}

\newcommand{\T}{\ensuremath{\mathbb{T}}}

\definecolor{grey}{rgb}{0.6,0.6,0.6}

\numberwithin{equation}{section}

\newtheorem{theorem}{Theorem}[section]
\newtheorem{lemma}[theorem]{Lemma}
\newtheorem{proposition}[theorem]{Proposition}

\theoremstyle{definition}

\newtheoremstyle{remarkstyle}
  {3pt}
  {3pt}
  {\small}
  {}
  {\bfseries}
  {.}
  { }
  {}

\theoremstyle{remarkstyle}

\title[Average free boundary velocity in the HS problem]{An efficient numerical method for estimating the average free boundary velocity in an inhomogeneous Hele-Shaw problem}

\author[I. Palupi]{Irma Palupi}
\address[I. Palupi] {Graduate School of Natural Science and Technology, Kanazawa University, Kakuma, Kanazawa, 920-1192, Japan}
\email{irmapalupi@telkomuniversity.ac.id}

\author[N. Po\v{z}\'{a}r]{Norbert Po\v{z}\'{a}r}
\address[N. Po\v{z}\'ar]{Faculty of Mathematics and Physics, Institute of Science and Engineering, Kanazawa University, Kakuma, Kanazawa, 920-1192, Japan}
\email{npozar@se.kanazawa-u.ac.jp }

\thanks{}

\date{\today\ (git: \gitauthsdate, \gitshash)}

\subjclass[2010]{}

\begin{document}
\begin{abstract}
We develop a numerical method to estimate the average speed of the free boundary in a Hele-Shaw problem with periodic coefficients in both space and time. We test the accuracy of the method and present a few examples. We show numerical evidence of flat parts (facets) on the free boundary in the homogenization limit.
\end{abstract}

\maketitle

\section{Introduction}

Let $K \subset \Rn$, $N \in \N$, be a nonempty closed set with a smooth boundary and let $\Omega_0 \supset K$ be an open set with a smooth boundary. In what we call the Hele-Shaw problem, we are to find the family of open sets $\set{\Omega_t}_{t \geq 0}$ that evolves with the outer normal velocity
\begin{align}
  \label{HS}
      V = g(x, t) |Du(x, t)| \qquad x \in \partial \Omega_t,
\end{align}
where $u = u(x,t)$ is at each time $t \geq 0$ the solution of the Laplace equation
\begin{align}
  \label{u-problem}
  \left\{\begin{aligned}
  -\Delta u(\cdot, t) &= 0 && \text{in } \Omega_t \setminus K \\
  u(\cdot, t) &= 0 && \text{on } \partial \Omega_t\\
  u(\cdot, t) &=1 && \text{on } \partial K,
  \end{aligned}\right.
\end{align}
and $g = g(x,t)$ is a given positive continuous function. Here $Du = (u_{x_1}, \ldots, u_{x_N})$ is the spatial gradient and $\Delta u = u_{x_1x_1} + \cdots + u_{x_Nx_N}$ is the Laplacian.

In two dimensions, the Hele-Shaw problem is a popular model of a flow of an incompressible fluid in between two close parallel plates known as the Hele-Shaw cell, see \cite{ST,Richardson}. It naturally generalizes to an arbitrary dimension and we refer to it by the same name in this paper. In three dimensions, in particular, it is a model of a flow of an incompressible fluid through a porous medium. The quantities have the following meaning: $u(x,t)$ is the pressure of the flowing fluid at a given point $x$ and a given time $t$, which fills the domain $\Omega_t$ at time $t$. $\partial \Omega_t$ is the interface between the fluid and the air, and it is a \emph{free boundary}. We neglect the surface tension effects and therefore we can normalize the pressure to $u = 0$ on $\partial \Omega_t$. Keeping the pressure at a prescribed positive value $1$ on $\partial K$ by injecting the fluid through $\partial K$, the fluid is advancing into the air. Darcy's law states that the velocity of the fluid is proportional to $-Du$. Since the free boundary is a level set of $u$, its normal velocity is proportional to $|Du|$. One can think about $\frac{1}{g(x,t)}$ as the depth of holes that the liquid must fill at the free boundary while it is advancing. We allow the depth to change with time.

\medskip

In recent years, there has been a lot of interest to understand the averaging behavior in evolutionary problems with oscillating coefficients in both space and time \cite{Schwab,Lin,JST}. In particular, the second author considered the homogenization of the Hele-Shaw problem \eqref{HS} in \cite{Pozar_ARMA}. It was also observed that non-periodic, fractal like variations in the flow lead to anomalous diffusion in Stefan and Hele-Shaw problems \cite{Voller}.

The goal of the homogenization approach is to understand how $g = g(x,t)$ influences the \emph{average} free boundary velocity. Clearly, we can observe an averaging behavior only if $g$ has a special structure, for example if $g$ is periodic. We investigate this in the \emph{homogenization limit}, that is, when the scale of these oscillations $\e \to 0$. Therefore we will assume that $g$ is periodic in both $x$ and $t$ with period $1$, that is
\begin{align*}
  g(x,t) = g(x + \xi, t + \tau) \quad \text{for all } \xi \in \Z^N, \tau \in \Z,
\end{align*}
and for scale $\e > 0$ we introduce the rescaled $g^\e(x,t) := g(\frac x\e, \frac t\e)$.
Note that in general scaling $g(\e^\alpha x, \e^\beta t)$ is possible, but it leads to a simpler behavior than this critical scaling $\alpha = \beta$, see for example \cite{Piccinini1}.
Keeping all other parameters fixed, for every given $\e > 0$ we get a solution $\set{\Omega^\e_t}_{t \geq 0}$, $u^\e$ of \eqref{HS} with $g = g^\e$. The goal is to identify the homogenization limit $\e \to 0$ of these solutions.

We shall denote
\begin{align*}
\Omega = \bigcup_{t \geq 0} \Omega_t \times \set t \quad \text{and}\quad \Omega^\e := \bigcup_{t \geq 0} \Omega_t^\e \times \set t.
\end{align*}
In \cite{Pozar_ARMA}, under certain regularity assumptions on the data $K$, $\Omega_0$ and $g$, it was proved that there exist limits $\set{\Omega_t}_{t\geq 0}$ and $u$ such that $u^\e \to u$ in the sense of half-relaxed limits and $\partial \Omega^\e \to \partial \Omega$ in Hausdorff distance. Furthermore, the pair $(\Omega, u)$ is the unique solution of the \emph{homogenized problem} in which $\Omega$ evolves with the normal free boundary velocity
\begin{align}
  \label{homogenized-HS}
  V = r(Du) \qquad x \in \partial \Omega_t,
\end{align}
and $u$ is again the solution of \eqref{u-problem}, where $r:\Rn \to \R$ is a nonnegative function that depends only on $g$.

A straightforward modification of the arguments in \cite{Pozar_ARMA} shows the same homogenization result if the Dirichlet boundary condition on $\partial K$ in \eqref{u-problem} is replaced by a Neumann boundary condition $\frac{\partial u}{\partial \nu}(\cdot, t) = 1$, where $\nu$ is the inner unit normal to $\partial K$.

However, there does not seem to be any explicit formula for $r(q)$, and it is not even known whether $r$ is continuous in general. It is only known that $r^*(a_1 q) \leq r_*(a_2 q)$ for any $0 < a_1 < a_2$ and any $q \in \Rn \setminus \set0$, where $r^*$ and $r_*$ denote the upper and lower semicontinuous envelopes of $r$, respectively. See \cite{Pozar_ARMA} for more details.
Formal calculations indicate that $r(q)$ is in general only $\frac 12$-H\"older continuous if $g$ is smooth and $r(q)$ might be discontinuous if $g$ is only H\"older \cite{KM_drops,Feldman}.
Our goal is to estimate $r(q)$ \emph{numerically}. We are in particular interested whether the homogenized problem has solutions whose free boundary develops flat parts (\emph{facets}). We propose an efficient numerical method to estimate $r(q)$ in dimension $N = 2$ and present some numerical results. Our method naturally generalizes to any dimension but we discuss only the two dimensional case for simplicity.

\subsection*{Outline} In the next section, we first discuss the one-dimensional setting to give a motivation for our numerical method to estimate $r(q)$ in two dimensions, which is then introduced in Section~\ref{sec:two-dimensions}. In Section~\ref{sec:numerical-results}, we present a few results of the numerical computation.

\section{The Hele-Shaw problem in one dimension}
\label{sec:one-dimension}

To motivate our work, let us briefly discuss the behavior of the Hele-Shaw problem \eqref{HS}--\eqref{u-problem} in one dimension.
Let
$K = (-\infty, 0]$ and $\Omega_0 = (-\infty, y_0)$,
and we use the boundary condition $u^\e_x(0, t) = q$ on $\partial K = \set{0}$ for some $q < 0$ in \eqref{u-problem}.
Then $\Omega^\e_t = (-\infty, y^\e(t))$ for some $y^\e > 0$.
The solution of Laplace's equation is $u^\e(x, t) = q(x - y^\e(t))$ in this case.
The free boundary velocity equation for $\Omega^\e$ simplifies to
\begin{align}
  \label{ode}
  \left\{\begin{aligned}
  (y^\e)'(t) &= g(\tfrac {y^\e(t)}\e, \tfrac t\e) |q|, \quad t > 0,\\
  y^\e(0) &= y_0,
  \end{aligned}\right.
\end{align}
which is a simple initial-value problem for an ordinary differential equation (ODE).
It is known, see \cite{Piccinini1,IM}, that $y^\e$ converges locally uniformly as $\e \to 0+$ to the
solution $y$ of the ODE
\begin{align*}
  \left\{\begin{aligned}
  y'(t) &= r(q), && t > 0,\\
  y(0) &= y_0,
  \end{aligned}\right.
\end{align*}
where $r: \R \to \R$ depends only on $g$. This equation has the unique solution $y(t) =
y_0 + t r(q)$. We can therefore \emph{estimate $r(q)$ numerically} by solving
\eqref{ode} for a small $\e > 0$ and finding
\begin{align*}
  r(q) = y(1) - y_0 \approx y^\e(1) - y_0.
\end{align*}
By a scaling argument, this can be shown equivalent to solving \eqref{ode} with $\e = 1$ for a large time $T \gg1$  and then finding
\begin{align*}
r(q) = \frac{y(T) - y_0}T \approx \frac{y^1(T) - y_0}T.
\end{align*}
Since \eqref{ode} can be efficiently solved numerically, we can estimate $r(q)$ rather easily.

The actual form of $r(q)$ is known only in certain cases \cite{Piccinini1}:
\begin{itemize}
  \item $g(x,t) = g(t)$: if $g$ is a $1$-periodic function of $t$ only, then $r(q) = \ang{g} |q|$, where $\ang{g} =
    \int_0^1 g(t) \diff t$ is the average of $g$.
  \item $g(x,t) = g(x)$: if $g$ is a $1$-periodic function of $x$ only, then $r(q) = \frac 1{\ang{\frac 1{g}}} |q|$,
    where $\ang{\frac 1g} =
    \int_0^1 \frac 1{g(x)} \diff x$ is the average of $\frac 1g$.
\end{itemize}

If $g$ depends on both $x$ and $t$ nontrivially, the explicit form of $r(q)$ is not known and in
fact it can be very complicated, see Figure~\ref{fig:rq} for an example. The number $r(q)$ is related to Poincar\'e's rotation number of a dynamical system corresponding to the ODE \eqref{ode}, see for instance \cite{JST} and the references therein.
\begin{figure}[t]
  \centering
  \includegraphics[width=0.7\textwidth]{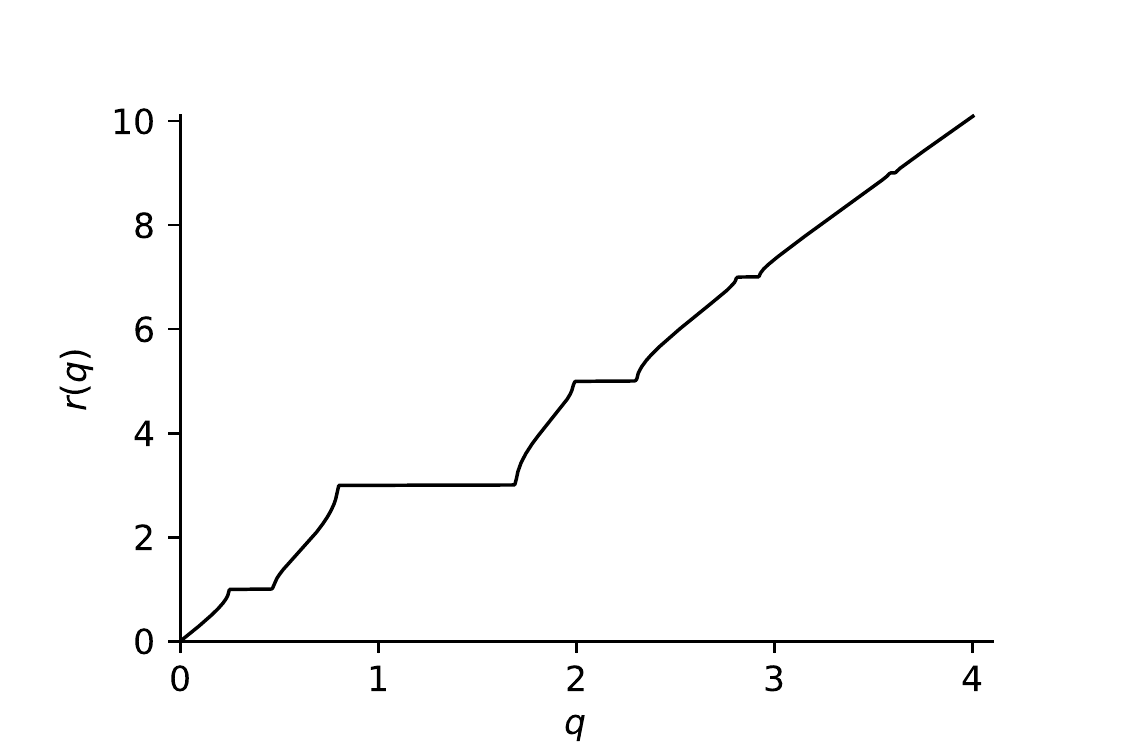}
  \caption{Sample $r(q)$ in one dimension for $g(x,t) = \sin(2\pi(x + t)) + \sin(2\pi(x + 3t)) + 3$. Note the pinning intervals at speeds $2k - 1$ for $k = 1, 2, \ldots, 5$. The graph was computed numerically using the method in Section~\ref{sec:one-dimension}.}
  \label{fig:rq}
\end{figure}

Nonetheless, we can still find the value of $r(q)$ at least for particular $q$. Particularly interesting is the existence of intervals of constant velocity, which we call \emph{pinning intervals}. See also \cite{KM_drops} for a related problem on a droplet motion.

\begin{proposition}
\label{pr:pinning-interval}
Suppose that $g(x,t) = f(x - t)$ where $f = f(x)$ is a positive periodic Lipschitz continuous function. Then $r(q) = 1$ for $q \in [-\frac{1}{\min f}, - \frac{1}{\max f}]$.
\end{proposition}

\begin{proof}
Let $L > 0$ be a period of $f$. Fix $q \in [-\frac{1}{\min f}, - \frac{1}{\max f}]$. Since $\frac{1}{|q|} \in [\min f, \max f]$, there exists $\xi \in \R$ such that $f(x_0) = \frac{1}{|q|}$ for all $x_0 \in \xi + L\Z$. But $y^\e(t) = \e x_0 + t$ is then a solution of \eqref{ode} for any $x_0 \in \xi + L\Z$ and $\e >0$. By uniqueness of \eqref{ode} (comparison principle), we conclude that $\frac{y^\e(T) - y^\e(0)}T = 1$ for any $\e >0$ and therefore $r(q) = 1$.
\end{proof}

In higher dimensions, if $g$ is time-independent, then it was shown in \cite{R1,KM_ARMA} that the solutions of \eqref{HS}
converge to the solutions of the homogenized problem \eqref{homogenized-HS} with $r(q) = \frac 1{\ang{\frac
1{g}}} |q|$ as in the one-dimensional case. A simple scaling argument shows that also $g = g(t)$ homogenizes to $r(q) = \ang{g} |q|$.

\section{Estimating the homogenized velocity in two dimensions}
\label{sec:two-dimensions}

In this section we propose a numerical method to estimate the homogenized velocity $r = r(q)$ in \eqref{homogenized-HS}, with a focus on dimension $N = 2$.
In contrast to the one-dimensional situation, the shape of the free boundary $\partial \Omega^\e_t$ is in general not flat in two dimensions and therefore the solution of \eqref{u-problem} is not a linear function anymore. We therefore have to solve the full problem to estimate $r(q)$. We first observe that for a given $q \in \Rn$ the moving plane
\begin{align*}
  P_q(x,t) &:= |q| \pth{r(q) t - x \cdot \frac q{|q|}} \qquad x \in \Rn, t \in \R,\\
  \Omega_q&:= \set{(x,t) : P_q(x, t) > 0} = \set{(x, t) :x \cdot \frac q{|q|} < r(q) t}
\end{align*}
satisfies the homogenized free boundary velocity law \eqref{homogenized-HS} with $u = P_{q,r}$ and $\Omega = \Omega_q$.

Let us suppose that $q = (q_1, 0)$ for some $q_1 < 0$. We consider the Hele-Shaw problem with $K := (-\infty, 0] \times \R \subset \Omega_0 := (-\infty, L_0) \times \R \subset \R^2$ for some fixed $L_0 > 0$, with Neumann boundary condition $u_{x_1}(0, x_2) = q_1$ for all $x_2 \in \R$. If we denote the canonical basis of $\R^2$ by $\{e_1,e_2\}$, clearly $(\Omega, u) = (\Omega_q +L_0 (e_1, 0), P_q(\cdot - L_0 e_1, \cdot))$ is a solution of the homogenized problem \eqref{homogenized-HS}--\eqref{u-problem} with the above initial and boundary data.
Let $\Omega^\e$, $u^\e$ be the solution of the $\e$-problem with the same boundary and initial data. By \cite{Pozar_ARMA}, we know that $\partial \Omega^\e \to \partial \Omega_q + L_0 (e_1, 0)$ in Hausdorff distance.
Let us fix $L_1 > L_0$ and define the first time the free boundary of the solution of the $\e$-problem touches the set $\set{x_1 = L_1}$,
\begin{align*}T_\e := \sup\set{t > 0: \Omega^\e_t \cap \set{x_1 = L_1} = \emptyset}.\end{align*}
By the convergence in the Hausdorff distance, we see that $T_\e \to \frac{L_1 - L_0}{r(q)}$ as $\e \to 0$.
This allows us to estimate $r(q)$ by choosing $0 < \e \ll 1$ and using
\begin{align*}
  r(q) \approx \frac{L_1 - L_0}{T_\e}.
\end{align*}

We will find $T_\e$ numerically by solving the problem on a bounded domain.
To this end, we observe that if $\e = \frac 1{\omega}$ for some $\omega \in \N$ sufficiently large, the uniqueness of solutions of the $\e$-problem implies that $\Omega^\e$, $u^\e$ are $1$-periodic in the $x_2$-direction, that is, $\Omega^\e + (e_2, 0) = \Omega^\e$, $u^\e(x + e_2, t) = u^\e(x,t)$.

Therefore we introduce the numerical domain $U = (0,1) \times \T$, where $\T = \R / \Z$ is the one-dimensional torus, and solve the Hele-Shaw problem on $U$ with boundary conditions
\begin{align*}
  \Omega^\e + (e_2,0) &= \Omega^\e,\\
  u_{x_1}(0, x_2, t) &= q_1 && 0 \leq x_2 \leq 1, t \geq 0,\\
  u(x_1, x_2 + 1, t) &= u(x_1, x_2, t) && x \in \Omega^\e_t, t \geq 0,
\end{align*}
see Figure~\ref{fig:domain}.
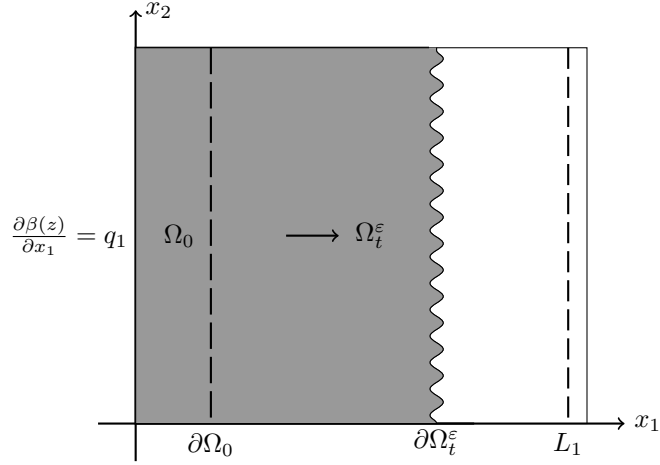
\begin{figure}[t]
	\centering
	\begin{tikzpicture}
		\draw[->, thick] (0,-0.5) -- (0,5.5) node[anchor=west] {$x_2$};
		\draw[->, thick] (-0.5,0) -- (6.5,0) node[anchor=west] {$x_1$};
		\draw (0,0) rectangle (6,5) ;
		\draw[name path=A, thick] (3.9,5)--(0,5)--(0,0)--(4.5,0);
		\draw[name path=B, decorate, decoration=snake] (4,0) -- (4,5);
		\draw[fill=grey, intersection segments= {of=A and B}];
		\draw [dash pattern = on8pt off3pt, thick] (5.75,5) -- (5.75, 0) node[below] {$L_1$};
		\draw [dash pattern = on10pt off3pt, thick] (1,5) -- (1,0) node[below] {$\partial \Omega_0$};
		\draw [](0.25,2.5) node[anchor=west] {$\Omega_0$};
		\draw[] (4,-0.25) node[] {$\partial \Omega^\e_t$};
		\draw [](-1.8,2.5) node[anchor=west] {$\frac{\partial \beta(z)}{\partial x_1} = q_1$} ;
		\draw [->, thick] (2,2.5) -- (2.7,2.5);
		\draw[] (1.5,2.5) node[text width = 3 cm, align = center, anchor = west] {$\Omega^\e_t$};
	\end{tikzpicture}
	\caption{Neumann problem \eqref{elliptic-problem} in 2D case to estimate the value of $r(q)$.}
  \label{fig:domain}
\end{figure}

There are direct methods to solve the Hele-Shaw problem, however, for simplicity and efficiency, we use the fact that the solution of the Hele-Shaw problem can be in the limit $\lambda \to 0$ approximated \cite{Lacey,LR} by the solution of the Stefan problem
\begin{align}
\label{u-Stefan}
\left\{\begin{aligned}
\lambda u_t - \Delta u &= 0 &&\text{in } \Omega,\\
                          V &= g^\e |Du| && \text{on $\partial\Omega$},
  \end{aligned}\right.
\end{align}
with initial condition $u(\cdot, 0) = u_0$, where $u_0$ is the $1$-periodic-in-$x_2$ solution of
\begin{align*}
\left\{\begin{aligned}
-\Delta u_0 &= 0 && \text{in }\Omega_0 \setminus K\\
u_0 &= 0 && \text{on } \partial \Omega_0,\\
\partial_{x_2} u_0 &= q_1 && \text{on } \set{x_2 = 0}.
\end{aligned}\right.
\end{align*}

This problem can be rewritten in the \emph{enthalpy} formulation by introducing $\beta(s) := \max(s, 0)$ and solving formally for $z: \Rn \times \R \to \R$ the solution of
\begin{align}
\label{z-Stefan}
\left\{\begin{aligned}
\lambda z_t - \Delta \beta(z) &= - \pth{\frac{\partial}{\partial t} \frac{1}{g^\e}} \chi_{\interior \set{z < 0}} && \text{in } U \times (0, \infty),\\
\frac{\partial \beta(z)}{\partial x_1}(0, x_2, t) &= q_1 && \text{for } x_2 \in [0,1), t > 0,\\
z &\text{\quad$1$-periodic in $x_2$,}\\
z(\cdot, 0) &= u_0 \chi_{\Omega_0} - \frac{1}{\lambda g^\e(\cdot, 0)} \chi_{\Omega_0^c} && \text{in } U.
\end{aligned}\right.
\end{align}
Here $\chi$ is the indicator function of a given set and $\interior \set{z < 0}$ is the interior of the set $\set{z < 0}$. The solution is understood in the sense of distributions. We can recover $u$ as $\beta(z)$ and $\Omega$ as $\set{z > 0}$.

If $g^\e = g^\e(x)$, the well-posedness of problem \eqref{z-Stefan} is well known from the theory of variational obstacle problems, see for example \cite{FK,R1}, and $u = \beta(z)$ is continuous \cite{CF}. We do not address the well-posedness when $g^\e = g^\e(x,t)$, but show at least
that \eqref{z-Stefan} is equivalent to \eqref{u-Stefan} with the same boundary data for classical solutions.

Let us therefore assume that there exists a differentiable function $s: K^c \to [0, \infty)$, $D s \neq 0$ for $t > 0$, such that $z \in C^2(Q_+) \cap C^1(\cl{Q_+}) \cap C^1(Q_-) \cap C(\cl{Q_-})$, $z > 0$ in $Q_+$, $z(x, s(x)) = 0$ if $s(x) > 0$, $z < 0$ in $Q_-$ where
\begin{align*}Q_\pm := \set{(x,t): x \in K^c, \ t > 0, \pm(t - s(x)) > 0}.\end{align*}
By $z \in C(\cl{Q_-})$ we understand the $z$ has a limit denoted as $z(x, s(x)-)$ as $(y,t) \to (x, s(x))$ along sequences with $t < s(y)$.

Assume that $z$ satisfies \eqref{z-Stefan} in the sense of distributions. Let us take a test function $\varphi \in C^\infty_c(K^c \times (0, \infty))$. We have
\begin{align*}
0&=\int_{K^c} \int_0^\infty \lambda z \varphi_t + \beta(z) \Delta \varphi - \pth{\frac{\partial}{\partial t} \frac{1}{g^\e}} \chi_{\interior \set{z < 0}} \varphi \dx \diff{t} =\\
& = \int_{Q^+} \beta(z) (\lambda\varphi_t + \Delta \varphi) \dx \diff{t} + \int_{Q^-} \lambda z \varphi_t - \pth{\frac{\partial}{\partial t} \frac{1}{g^\e}} \varphi \dx \diff{t} := I_+ + I_-.
\end{align*}
Integration by parts on the individual terms $I_\pm$ yields
\begin{align*}
I_+ &= \int_{K^c} \int_{s(x)}^\infty \lambda \beta(z) \varphi_t \diff t \dx + \int_0^\infty \int_{\set{s(x) < t}} \beta(z) \Delta \varphi \dx \diff t\\
&= -\int_{Q^+} (\lambda \partial_t \beta(z) - \Delta \beta(z)) \varphi \dx \diff t - \int_0^\infty \int_{\set{s(x) = t}} D\beta(z) \cdot \frac{Ds}{|Ds|} \varphi \diff{\mathcal H^{n-1}} \diff t,
\end{align*}
where we used that $z(x, s(x)) = 0$ and that the unit outer normal vector to $\set{x: s(x) < t}$ is $\frac{Ds}{|Ds|}$, and
\begin{align*}
I_- &= \int_{Q_-} \pth{-\lambda z_t  -  \pth{\frac{\partial}{\partial t} \frac{1}{g^\e}}} \varphi \dx \diff t
+ \int_{K^c} \lambda z(x, s(x)-) \varphi(x, s(x)) \dx.
\end{align*}
From this we immediately have that $u = \beta(z)$ satisfies $\lambda u_t - \Delta u = 0$ in $Q_+$, and $z = -\frac 1{\lambda g^\e}$ in $Q_-$. In particular, $z(x, s(x)-) = -\frac 1{\lambda g^\e(x, s(x))}$. The coarea formula yields
\begin{align*}
\int_0^\infty \int_{\set{s(x) = t}} D \beta(z) \cdot \frac{Ds}{|Ds|} \varphi \diff{\mathcal{H}^{n-1}} \diff t = \int_{K^c} D \beta(z)(x,s(x)) \cdot Ds(x,s(x))  \varphi(x,s(x)) \dx.
\end{align*}
Therefore
\begin{align*}
\int_{K^c} \pth{-\frac{\lambda}{\lambda g^\e(x, s(x))} -D u(x, s(x)) \cdot D s(x) }\varphi(x,s(x)) \dx = 0.
\end{align*}
But $V = \frac{1}{|Ds|}$ and $- D u \cdot Ds = |Du||Ds|$ and therefore we conclude that
\begin{align*}
V =  g^\e|Du|.
\end{align*}

A numerical solution of problem \eqref{z-Stefan} can be found efficiently by the method introduced by Berger, Br\'ezis and Rogers \cite{BBR}, in the form further studied by Murakawa \cite{Murakawa}. We refer to this scheme as the BBR scheme.
Choosing a time step $\tau > 0$, we iteratively find the sequences $\set{u^k}_{k \geq 1}$, $\set{z^k}_{k \geq 0}$ of solutions of
\begin{subequations}
\begin{align}
\label{u-update}
  &\left\{\begin{aligned}
  \lambda \mu^{k-1} u^k - \tau \Delta u^k &= \lambda \mu^{k-1} \beta(z^{k-1}) &&\text{in } U,\\
  \frac{\partial u^k}{\partial x_1}(0, \cdot) &= q_1,\\
  u^k(1, \cdot) &= 0,\\
  u^k & \qquad \text{$1$-periodic in $x_2$},
\end{aligned}\right.\\
\label{z-update}
  z^k &= z^{k-1} + \mu^{k-1} (u^k - \beta(z^{k-1})) - \frac \tau\lambda \pth{\frac{\partial}{\partial t} \frac{1}{g^\e}}(\cdot, t_{k-\frac 12}) \chi_{\interior \set{z^{k-1} < 0}},\\
\label{mu-update}
  \mu^k &= \frac 1{\delta + \beta'(z^k)},
\end{align}
\end{subequations}
for $k = 1, 2, \ldots$, with $z^0 := z(\cdot, 0)$.
Here $\delta > 0$ is a chosen regularization parameter that we discuss below, and we define
\begin{align*}
  \beta'(s) :=
  \begin{cases}
    1, & s >0,\\
    0, & s \leq 0.
  \end{cases}
\end{align*}

Note that we add the source $- \frac1\lambda\pth{\frac{\partial}{\partial t} \frac{1}{g^\e}}(\cdot, t_{k-\frac 12}) \chi_{\interior \set{z^{k-1} < 0}}$, $t_{k-\frac 12} = (k - \frac 12) \tau$, to the update of $z$ in \eqref{z-update} rather than the problem \eqref{u-update} as was done in \cite{BBR}. This is to avoid any unwanted diffusion in $\set{z < 0}$ that would otherwise occur.

Let us comment on the choice of $\tau$ and $\delta$. The time step restriction comes from the fact that the free boundary can advance at most distance $h$ (one node distance) in one time step. We therefore take the time step $\tau < \frac h{2 V_{\rm max}}$, where $V_{\rm max}$ is some reasonable estimate on the maximum velocity of the free boundary in the problem.

The maximum principle yields $u^k > 0$ on $U$. The regularization parameter $\delta > 0$ guarantees a presence of a boundary layer in the neighborhood of the free boundary where $u^k$ is sufficiently large so that $z$ is increasing there. This boundary layer limits the resolution with which the function $g^\e := g(\frac\cdot\e, \frac\cdot\e)$ is resolved and hence we need to control its size. Let us estimate its width. Assuming a one-dimensional situation for simplicity, with free boundary position of $z^{k-1}$ located at $x_1 = 0$ with $z^{k-1} < 0$ for $x_1 > 0$, \eqref{u-update} simplifies in $\set{x_1 > 0}$ to
\begin{align*}
\frac{\lambda}{\delta} u^k - \tau u^k_{x_1x_1} = 0,
\end{align*}
and therefore the dominating term in the solution will be $\phi(x_1) = C e^{-\sqrt{\frac{\lambda}{\delta \tau}}x_1}$, where $C \approx - q_1 \sqrt{\frac{\delta \tau}{\lambda}}$ so that the derivative $\phi_{x_1}$ is approximately $q_1$ at $x_1 = 0$.

From \eqref{z-update}, the total amount of energy deposited into the negative $z$ per one time step is therefore $\int_0^\infty z^k - z^{k-1} \dx_1 = \frac1\delta\int_0^\infty \phi \dx_1 \approx -q_1 \frac\tau\lambda$, which is as expected from Fourier's law. We need to choose $\delta > 0$ so that the majority is deposited near the free boundary $\set{x_1 = 0}$. The ratio deposited in $\set{a < x_1}$ for some $a > 0$ is given by $\int_a^\infty \phi \dx_1 / \int_0^\infty \phi \dx_1 = e^{-\sqrt{\frac{\lambda}{\delta\tau}}a}$. For this to be equal to a given $\gamma \in (0,1)$ with $a = wh$, $w > 0$, we need to take
\begin{align*}
\delta = \pth{\frac w{\log \gamma}}^2 \frac{\lambda h^2}\tau.
\end{align*}
We have not observed any ill effect if we choose small $w$, so we in general set $w = 1$, $\gamma = 0.01$, which yields the formula
\begin{align}
\label{regularization-delta}
\delta \approx 4.7 \times 10^{-2} \frac{\lambda h^2}{\tau}.
\end{align}

Let us explain how we implement the update of $z$ in the set $\set{z <0}$ in the BBR method \eqref{z-update}. Since the set $\set{z > 0}$ is monotonically increasing in time if $z$ is the exact solution of \eqref{z-Stefan}, $z = -\frac{1}{\lambda g^\e}$ in $\set{z < 0}$. However, in the BBR method \eqref{z-update}, the value of $z^k$ is also influenced in $\set{z^{k-1} <0}$ by $u^k$ since $u^k > 0$ in $U$ by the comparison principle. On the other hand, $u^k$ decreases exponentially with the distance from $\set{z^{k-1} > 0}$ as observed above. We therefore make use of this fact and at a given fixed point $x$, we set $z^k(x) = -\frac{1}{\lambda g^\e(x, t_k)}$ at the first time step $k$ such that $u^k(x) > 10^{-3} \delta$, where $\delta$ is the regularization parameter in \eqref{mu-update}, and only at the later time steps we apply the update \eqref{z-update} at this node. This leads to a significant increase in the accuracy of the estimate of $r(q)$ as tested in Section~\ref{sec:numerical-results}, especially in a neighborhood of the pinning intervals. Note that \eqref{z-update} is approximately equivalent in $\set{z^{k-1} < 0} \cap \set{u^k \ll \delta}$ to a second-order accurate numerical integration of the derivative of $-\frac 1{\lambda g^\e}$. Used directly, it leads to a large error over a few periods of $g^\e$. But the pinning interval, and the value $r(q)$ in general, is very sensitive to $\max g^\e$ and $\min g^\e$ as indicated by Propostion~\ref{pr:pinning-interval}. Hence to reach a reasonable accuracy, we use the numerical integration only inside the boundary layer by implementing the above scheme.

\bigskip

Since we need to find the solutions of the Hele-Shaw problem for many different $q$ over a large time interval (relative to $\e$) to get a reliable estimate on $r(q)$, it is important to develop an efficient numerical method to solve the elliptic problem \eqref{u-update} for $u^k$.
It turns out that a multigrid scheme for the linear elliptic problem for $u^k$ works well even though $\mu^{k-1}$ has a jump across the free boundary.

\subsection{The multigrid method}

In this section we focus on the linear elliptic problem
\begin{align}
\label{elliptic-problem}
\left\{\begin{aligned}
a u - h^2 \Delta u &= f &&\text{in }U = (0,1) \times \T\\
  u_{x_1}(0, x_2) &= q_1 &&\text{for } x_2 \in [0,1)\\
  u(1, x_2) &= 0 &&\text{for } x_2 \in [0,1)\\
  u & \qquad \text{$1$-periodic in $x_2$},
\end{aligned}\right.
\end{align}
where $h > 0$ will be the discretization step, $a$ is a given bounded nonnegative function, and $f$ is a given bounded function.

We will choose $M = 2^p$ for some $p \in \N$ as the resolution and set $h = \frac{1}{M}$ and introduce $x_i = i h$, $i = 0, \ldots, M$.
We discretize the PDE using the standard finite difference method with the central difference on a 5-point stencil.
We therefore look for $v_{i,j}$, $i,j = 0, \ldots, M-1$, that approximate $u(x_i, x_j)$.
For the Neumann boundary condition, we use a ghost grid point assuming $v_{-1, j} = v_{1, j} - 2 q_1 h$.
This leads to the linear system
\begin{align}
\label{linear-system}
\left\{\begin{aligned}
  (4 + a_{i, j}) v_{i,j} - v_{i-1, j} - v_{i+1, j} - v_{i, j-1} - v_{i, j+1} &= f_{i,j}, \\
  &i = 1, \ldots, M-1, j = 0, \ldots, M-1,\\
  (4 + a_{i, j}) v_{0,j} - 2v_{1, j} - v_{0, j-1} - v_{0, j+1} &= f_{0,j} - 2q_1h, \\&j = 0, \ldots, M-1,\\
  v_{i, -1} &= v_{i, M-1},\\
  v_{i, M} &= v_{i, 0},\\
  v_{M, j} &= 0,\\
\end{aligned}\right.
\end{align}
where $f_{i, j} := f(x_i, x_j)$ and $a_{i,j} := a(x_i, x_j)$.

To solve this system, we use the standard multigrid method, see for example \cite{Thomas}. We introduce a sequence of spaces
\begin{align*}
V^{2^mh} &= \R^{(M/2^m - 1) \times (M/2^m-1)}, \qquad m = 0, \ldots, p.
\end{align*}
On each of these we will solve the linear system with appropriately adjusted $M$.
We need to introduce the grid transfers.
The restriction operator $I_h^{2h} : V^h \to V^{2h}$ is defined by the standard weighted sum
\begin{align*}
(I_h^{2h} v^h)_{i,j} &=
\frac{v^h_{2i,2j}}4 +\frac{v^h_{2i-1,2j} + v^h_{2i+1,2j} + v^h_{2i,2j-1} + v^h_{2i,2j+1}}{8}\\
&\quad +\frac{v^h_{2i-1,2j-1} + v^h_{2i+1,2j-1} + v^h_{2i-1,2j+1} + v^h_{2i+1,2j+1}}{16},
\end{align*}
where we assume the periodic extension in $j$ and we assume that $v^h$ is even across $i = 0$, that is,
\begin{align*}
  v^h_{-1,j} &= v^h_{1, j},\\
  v^h_{i, -1} &= v^h_{i, M-1},\qquad
  v^h_{i, M} = v^h_{i, 0}.
\end{align*}
Indeed, the error correction will satisfy the boundary condition $u_{x_1}(0, x_2) = 0$.

The prolongation operator $I_{2h}^h: V^{2h} \to V^h$ is the standard prolongation
\begin{align*}
(I_{2h}^{h} v^{2h})_{2i,2j} &= v^{2h}_{i,j},\\
(I_{2h}^{h} v^{2h})_{2i+1,2j} &= \frac12\pth{v^{2h}_{i,j} + v^{2h}_{i+1,j}},\\
(I_{2h}^{h} v^{2h})_{2i,2j+1} &= \frac12\pth{v^{2h}_{i,j} + v^{2h}_{i,j+1}},\\
(I_{2h}^{h} v^{2h})_{2i+1,2j+1} &= \frac14\pth{v^{2h}_{i,j} + v^{2h}_{i,j+1} + v^{2h}_{i+1,j} + v^{2h}_{i+1,j+1}},
\end{align*}
again assuming the periodic extension $v^{2h}_{i, M} = v^{2h}_{i, 0}$.

By $A^{2^mh}$ we will denote the matrix of the linear system \eqref{linear-system} for grid with resolution $M/2^m$ with $a^{2^m h}$ defined recursively as
\begin{align*}
a^{2h} = 4 I_h^{2h} a^h.
\end{align*}

We perform the following multigrid V-cycle:

\begin{enumerate}
\item$ k_1$ times iterate the smoother for $A^h v^h = b^h$ with initial
  guess $v^{h, (0)}$, obtaining $v^{h,(k_1)}$.

\item Find the residual $r^h = b^h - A^h v^{h,(k_1)}$

\item Restrict the right-hand side $b^{2h} = 4 I_h^{2h} r^h$.

\item Solve $A^{2h} e^{2h} = b^{2h}$ on a half-resolution grid recursively.

\item Correct the approximation $\tilde v^{h, (k_1)} = v^{h, (k_1)} + I_{2h}^h
  e^{2h}$.

\item $k_2$ times iterate the smoother for $A^h v^h = b^h$ with initial
  guess $\tilde v^{h, (k_1)}$.
\end{enumerate}
The problem $A^{1}  e^{1} = b^{1}$ is solved exactly.

To improve the convergence, we solve for $e^{2h}$ using two V-cycles, with initial guess $e^{2h} = 0$.

As a smoother we implement the damped Jacobi method with damping constant $\omega = \frac 23$. We perform $k_1 = k_2 = 4$ relaxation iterations.
This reduces the maximum norm of the residual by about a factor of $10$ per iteration, see Table~\ref{tab:multigrid-residual}.
\begin{table}
\begin{center}
\begin{tabular}{c|ccccc}
iteration $k$ & 0 & 1 & 2 & 3 & 4\\
\hline
   $\norm{r^{(k)}}$&
  $1.95\times 10^{-3}$&
  $3.11\times 10^{-5}$&
  $1.54\times 10^{-6}$&
  $8.95\times 10^{-8}$&
  $5.53\times 10^{-9}$
\end{tabular}
\end{center}
\caption{Evolution of the residual in the multigrid method with $M = 1024$, $h = 1 /M$, $b^h = 0$, $a_{i,j} = 1000h^2$ if $x_i + 0.1 \sin(6\pi x_j) > 0.5$ and $a_{i,j} = h^2$ otherwise, and $f_{i,j} = 0$, with initial guess $v^h = 0$.}
\label{tab:multigrid-residual}
\end{table}

To estimate the time complexity, we observe that the matrix multiplication and the application of prolongation and restriction operators have each approximately the time complexity of a single Jacobi iteration. With these parameters, a simple estimate places the time complexity of the V-cycle at about 22 Jacobi iterations. Moreover, the method is parallelizable in a straightforward manner. We did not explore this point since we need to run a large number of computations and therefore can take advantage of a process-level parallelism.

\subsection{Application of the multigrid solver to the BBR scheme}

To find $u^k$ in \eqref{u-update}, we apply the above multigrid solver to \eqref{elliptic-problem} with $a = \frac{\lambda h^2}{\tau} \mu^{k-1}$ and $f = \frac{\lambda h^2}{\tau} \mu^{k-1} \beta(z^{k-1})$, and we use $u^{k-1}$ as the initial guess. In our computations it is generally sufficient to perform a fixed number of V-cycles per time step. We perform in general 1 to 3 V-cycles.

\subsection{General direction}\label{rotation}
\label{sec:general-direction}

So far we have assumed that $q = (q_1, 0)$ with $q_1 < 0$. To handle general $q \in \R^2 \setminus \set0$, we rotate the coordinate system so that $q$ is of this form. That is, instead of $g$ we consider
\begin{align*}
\tilde g(x, t) = g(x_1 \zeta + x_2 \zeta^\perp, t),
\end{align*}
where $\zeta^\perp = (-\zeta_2, \zeta_1) = \pth{\frac{q_2}{|q|}, -\frac{q_1}{|q|}}$.

Of course, in general $\tilde g$ is not periodic in $x_2$, unless there exist integers $n_1, n_2 \in \Z$, $n_1 n_2 \neq 0$, such that $n_1 q_1 + n_2 q_2 = 0$, that is, unless $q$ is a \emph{rational} direction. These are however the only directions that we can consider numerically.

By taking $\sigma = \frac{q_1}{n_2} = -\frac{q_2}{n_1}$ and $m_1 = n_2$, $m_2 = -n_1$, it can be easily seen that $q$ is a rational direction if and only if there exist $\sigma > 0$ and two integers $m_1, m_2 \in \Z$ such that $q = (m_1 \sigma, m_2\sigma)$. Let us show how we can choose $\e$ so that the solution of the $\e$-problem is $1$-periodic in the $x_2$ direction. To this end, we shall find the minimal period of $\tilde g$ first. This is equivalent to finding the smallest $s > 0$ such that $s \zeta^\perp \in \Z^2$.

\begin{lemma}
\label{le:}
If $m_1$ and $m_2$ are coprime, then $s = (m_1^2 + m_2^2)^{\frac12}$ is the smallest $s > 0$ such that $s \zeta^\perp \in \Z^2$.
\end{lemma}

\begin{proof}
Note that $\zeta^\perp = \frac{(m_2, -m_1)}{(m_1^2 + m_2^2)^{\frac12}}$. Clearly
\begin{align*}
s\zeta^\perp = (m_2, -m_1) \in \Z^2.
\end{align*}

Now suppose that there is $0 < \tilde s < s$ such that $\tilde s \zeta^\perp \in \Z^2$. But then $\frac{\tilde s}s s \zeta^\perp = \frac{\tilde s}{s} (m_2, -m_1) \in \Z^2$. In particular $\frac{\tilde s}{s} \in \Q$. Suppose that $\frac{\tilde s}{s} = \frac{p}{q}$, where $p, q$ are coprime. Since $\frac pq < 1$, $q > 1$ is a divisor of both $m_1$ and $m_2$. But that is a contradiction with $m_1$ and $m_2$ being coprime.
\end{proof}

Given a general $q = (m_1 \sigma, m_2 \sigma)$,
it is therefore sufficient to choose
\begin{align*}\e = \frac{\operatorname{gcd}(m_1, m_2)}{d(m_1^2 +m_2^2)^\frac12}\end{align*}
for some integer $d \in \N$, where $\operatorname{gcd}$ stands for the greatest common divisor, and the solution will be $1$-periodic in the $x_2$ direction. Note that this limits the angular resolution of our method. For example, near the $x_1$-axis, to compute $q = (m_1 \sigma, \sigma)$ near a fixed $\hat q = (\hat q_1, 0)$, we must take $\e < \frac 1{m_1}$ which requires large resolution $M$ for small $\sigma$ since $m_1 \approx \frac{\hat q_1}{\sigma}$.

\section{Numerical results}
\label{sec:numerical-results}

To test the numerical method, we estimate the homogenized velocity $r(q)$ for a few simple functions $g$. Namely, we consider
\begin{subequations}
\label{g-choices}
\begin{align}
\label{g-one}
g(x,t) &= \sin(2\pi(x_1 + t)) + 2,\\
\label{g-two}
g(x,t) &= \sin(2\pi(x_1 + t)) + \sin(2\pi(x_2 + t)) + 3,\\
\label{g-four}
g(x,t) &= \frac 12 \cos(2\pi t) \Big(\sin(2\pi x_1) + \sin(2\pi x_2)\Big) + 2,\\
g(x,t) &= \sin(2\pi(x_1 + t)) + \sin(2\pi(x_1 + 3 t)) + 3.
\end{align}
\end{subequations}
By Proposition~\ref{pr:pinning-interval}, the pinning interval with $r(q) = 1$ in \eqref{g-one} is $[\frac 13, 1] \times \set{0}$. Note that while \eqref{g-four} has the form of a standing wave, it is in fact a superposition of four traveling waves moving with speed $1$ in directions $(\pm1,0)$ and $(0,\pm1)$.

We always take $\lambda = 10^{-7}$, $\tau = \frac h8$ and $\delta$ as in \eqref{regularization-delta}, and $L_0 = 0.1$, $L_1 = 0.9$. The values of $r(q)$ are estimated for a range of $q = (m_1 \sigma, m_2 \sigma)$, with $\sigma = \frac{6.4}{M}$, $m_1, m_2 \in \Z$. For given $q$, we determine $\e$ following Section~\ref{sec:general-direction} as
\begin{align*}
\e = \frac 1{d (m_1^2 + m_2^2)^{\frac 12}}, \qquad d = \max \pth{1, \operatorname{round}\pth{\frac{9M}{64(m_1^2 + m_2^2)^{\frac12}}}}.
\end{align*}
This is done so that neighboring points have similar $\e$. Values $\sigma$ smaller than the above lead to high frequency oscillations in the estimate of $r(q)$ since $\e$ is then forced to be too small in proportion to $h = \frac 1M$. We always use $2$ V-cycles per time step, unless otherwise noted. These parameters produce very consistent results across a wide range of resolutions $64 \leq M \leq 1024$ that we tested, see Figures~\ref{fig:rg1}--\ref{fig:rg4}. In our numerical tests, the value of $\lambda$, if it is chosen sufficiently small, appears to have a negligible influence on the results, well within the errors reported in Table~\ref{tab:accuracy} for example.

The computational time necessary to estimate a single $r(q)$ is $O(M^3)$, and to produce a contour plot with the above resolution $\sigma$ is $O(M^5)$.

\subsection{Discussion}

We observed a number of pinning intervals for a few examples of coefficients $g$. The behavior of $r(q)$ in a neighborhood of the pinning intervals is surprisingly consistent across our computations. Namely, the velocity is pinned to a constant value only along a single critical direction, and far away from the pinning interval the value $r(q)$ is proportional to $|q|$ as in the time-independent case. Moreover, $r(q)$ appears to be only Lipschitz continuous at the points on the relative interior of the pinning interval, see Figure~\ref{fig:pinningq1}. For the critical direction, this has a boosting effect for smaller $|q|$ and slowing effect for larger $|q|$ along the pinning interval, compared to the nearby directions. We observed that this leads to an appearance of a stable flat part (facet) of the free boundary  in the critical direction, see Figure~\ref{fig:facet}.

In the particular case \eqref{g-two}, there is an indication of the appearance of a whole class of pinning intervals near the main diagonal $(1,1)$ as a sort of resonance between the pinning intervals in directions $(1,0)$ and $(0,1)$, see Figure~\ref{fig:rg2}. Some of the level sets in the first quadrant have an appearance of the level set of the $\ell^1$ norm $\norm{p}_1 := |p_1| + |p_2|$, which is a typical example of a so-called crystalline anisotropy. Somewhat surprisingly, no such effect is apparent in the case \eqref{g-four} with four traveling waves in the axis directions, see Figure~\ref{fig:rg4}.

Our estimate of $r(q)$ appears to be first order accurate in $M$, see Table~\ref{tab:accuracy}, and consistent when changing the resolution and other parameters, see Figure~\ref{fig:rg1} and Figure~\ref{fig:facet}.

\begin{figure}
\centering
	\subfloat{\includegraphics[clip, trim=0.5cm 2.cm 0.5cm 3.cm, width=\textwidth]{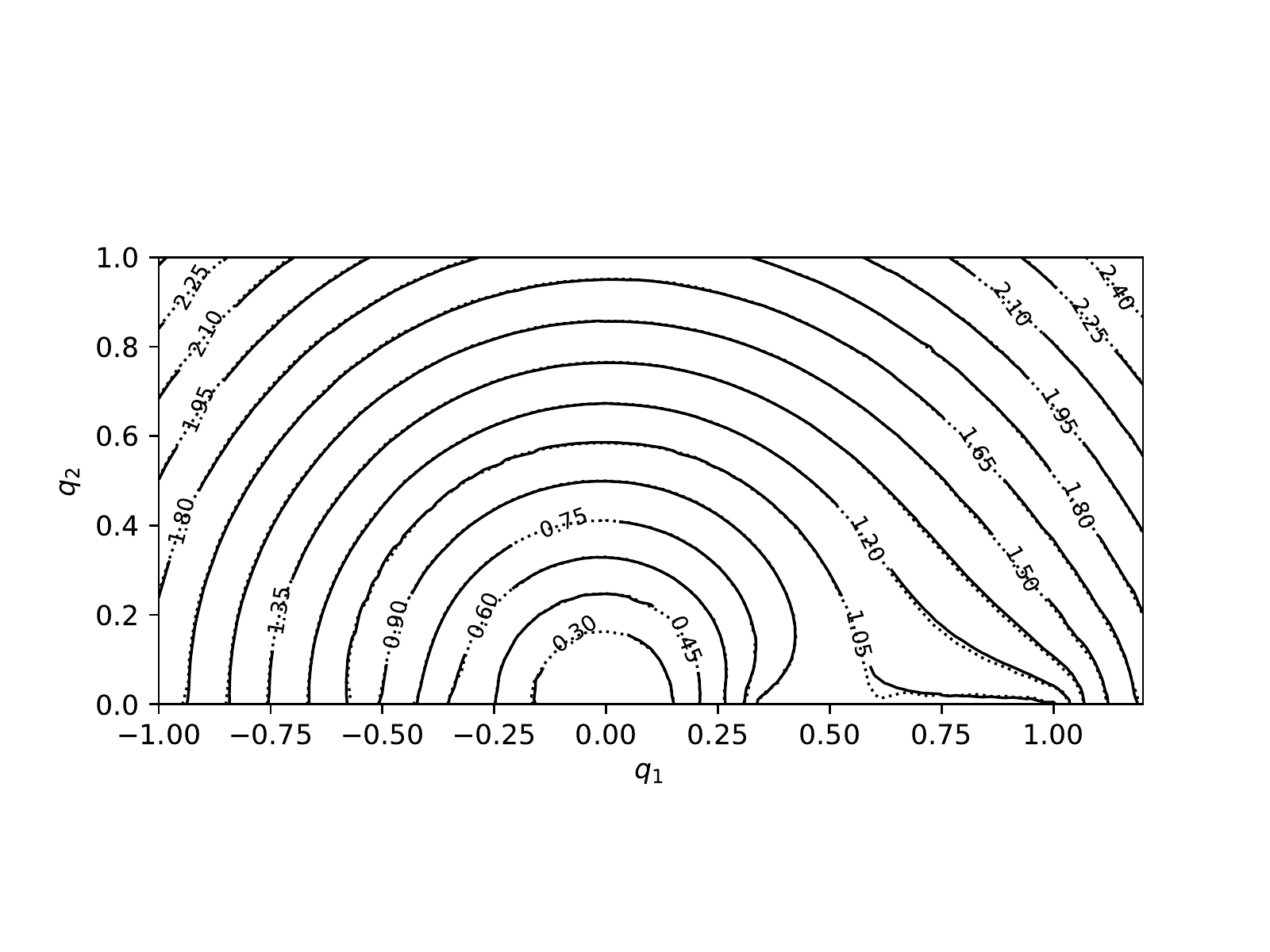}\label{fig:rg1_res}}\\
	\subfloat{\includegraphics[clip, trim=0.5cm 4.2cm 0.5cm 5.cm, width=\textwidth]{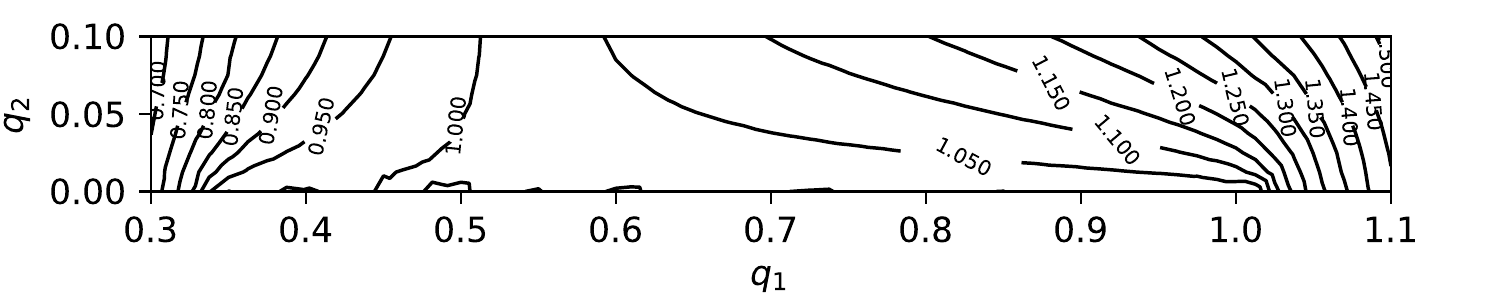}\label{fig:rg1_high}}
\caption{(Top) The contour plot of $r(q)$ with $g(x,t)= \sin (2 \pi (x_1+t)) +2$. The pinning interval $[\frac 13, 1] \times \set{0}$ is apparent, see Proposition~\ref{pr:pinning-interval}, where the average velocity is pinned to $1$. The solid contours were obtained with $M = 256$, while the dotted contours were obtained with $M  = 128$.
	(Bottom) Detail of the pinning interval computed with $M = 512$.}
\label{fig:rg1}
\end{figure}

\begin{figure}
	\centering
	\includegraphics[clip, trim=0.cm 0cm 0.cm .5cm, width=0.75\textwidth]{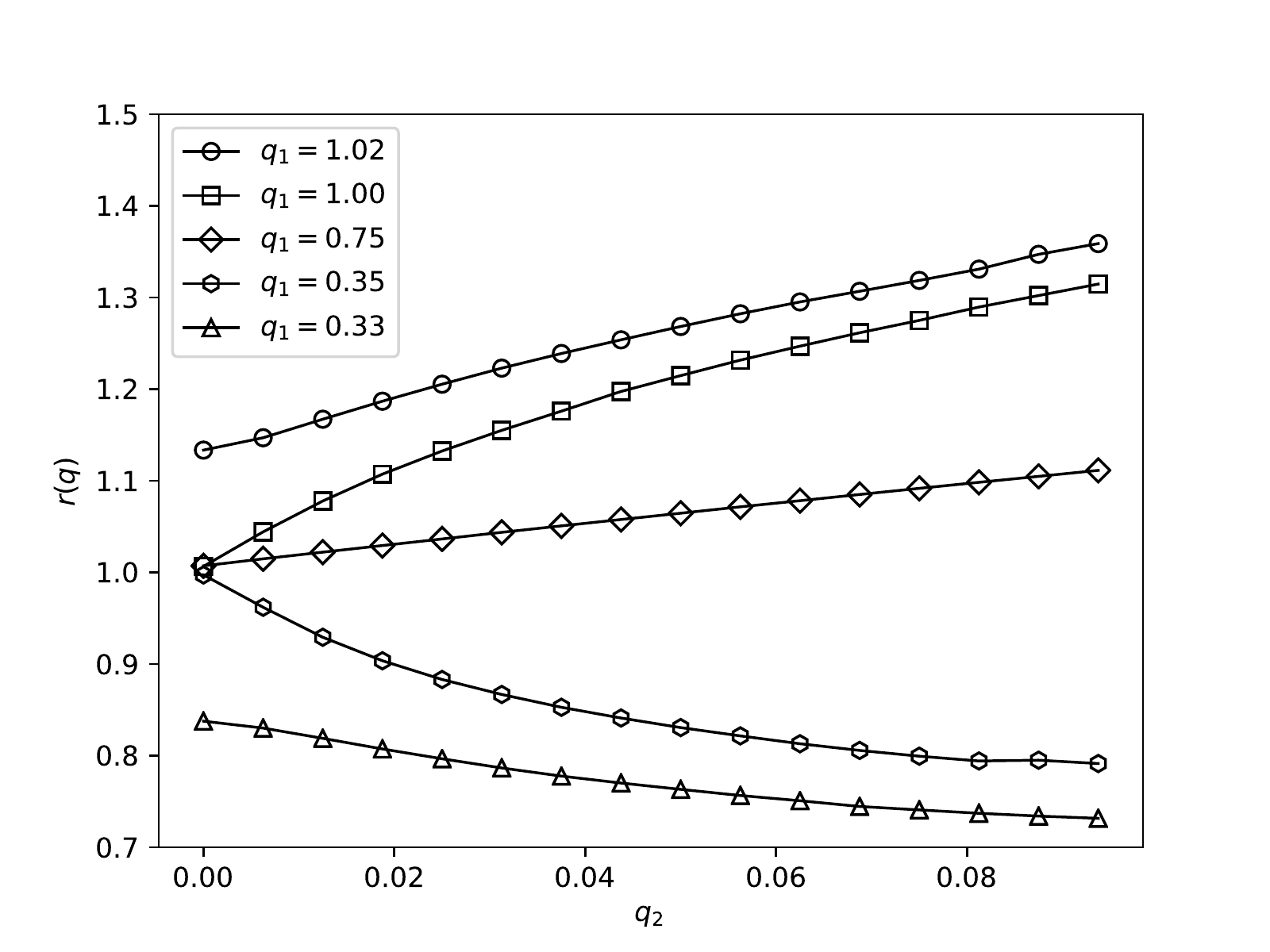}
	\caption{Values of $r(q)$, $q = (q_1, q_2)$, for $g(x,t)= \sin (2 \pi (x_1+t)) +2$ with $M = 1024$ as a function of $q_2$ for several chosen of $q_1$.}
	\label{fig:pinningq1}
\end{figure}

\begin{figure}
	\centering
  \subfloat{\includegraphics[clip, trim=0.5cm 1.5cm 0.5cm 2.5cm, width=\textwidth]{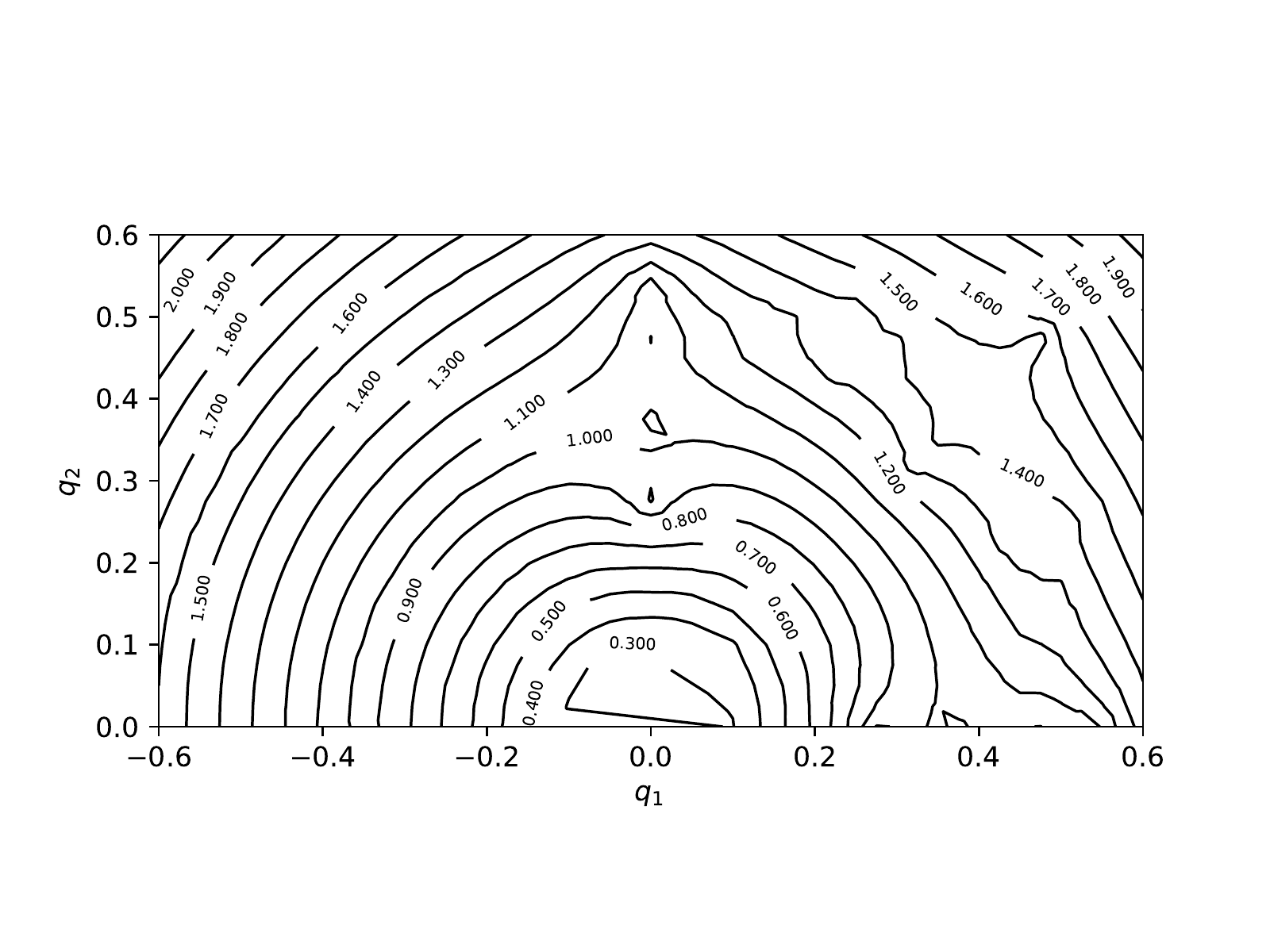}}\\
	\subfloat{\includegraphics[clip, trim=0.5cm 0.0cm 0.5cm 1cm,  width=\textwidth]{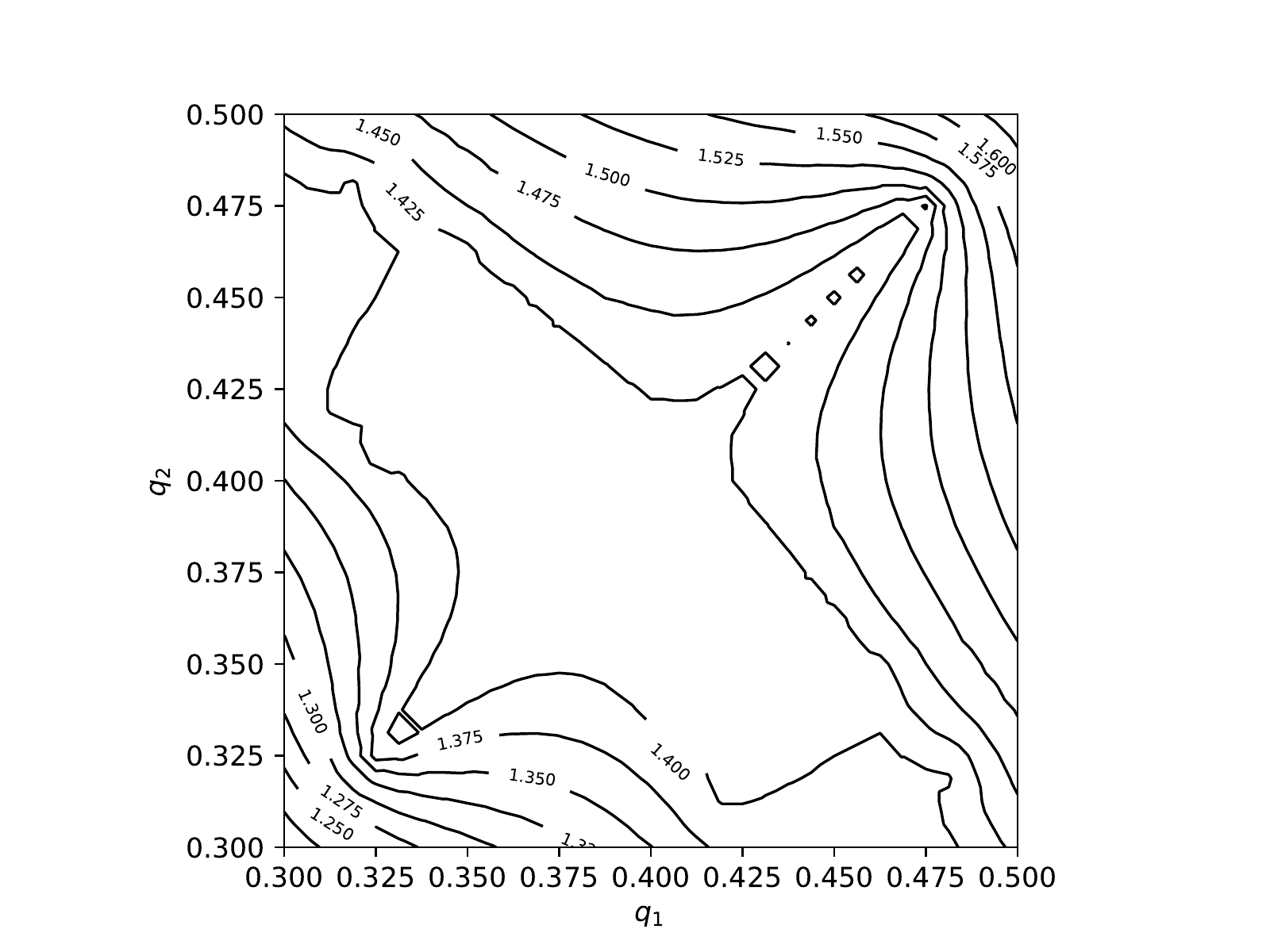}}
	\caption{(Top) Contour plot of $r(q)$ with $g(x,t)= \sin (2 \pi (x_1+t)) + \sin (2 \pi (x_2+t)) +3$ and $M = 256$. As expected, pinning intervals appear in directions $(1,0)$ and $(0,1)$, however, there is also a visible pinning interval in direction $(1,1)$ where the velocity appears to be pinned to $\sqrt{2}$. The plot is symmetric with respect to the reflection across the direction $(1,1)$.
	(Bottom) Detail with $M = 1024$ around the pinning interval with velocity $\sqrt{2}$. There might be other pinning intervals nearby. The small squares along the diagonal are artifacts of contour reconstruction.}
	\label{fig:rg2}
\end{figure}

\begin{figure}
	\centering
	\includegraphics[clip, trim=0.5cm 3.cm 0.5cm 3.5cm, width=\textwidth]{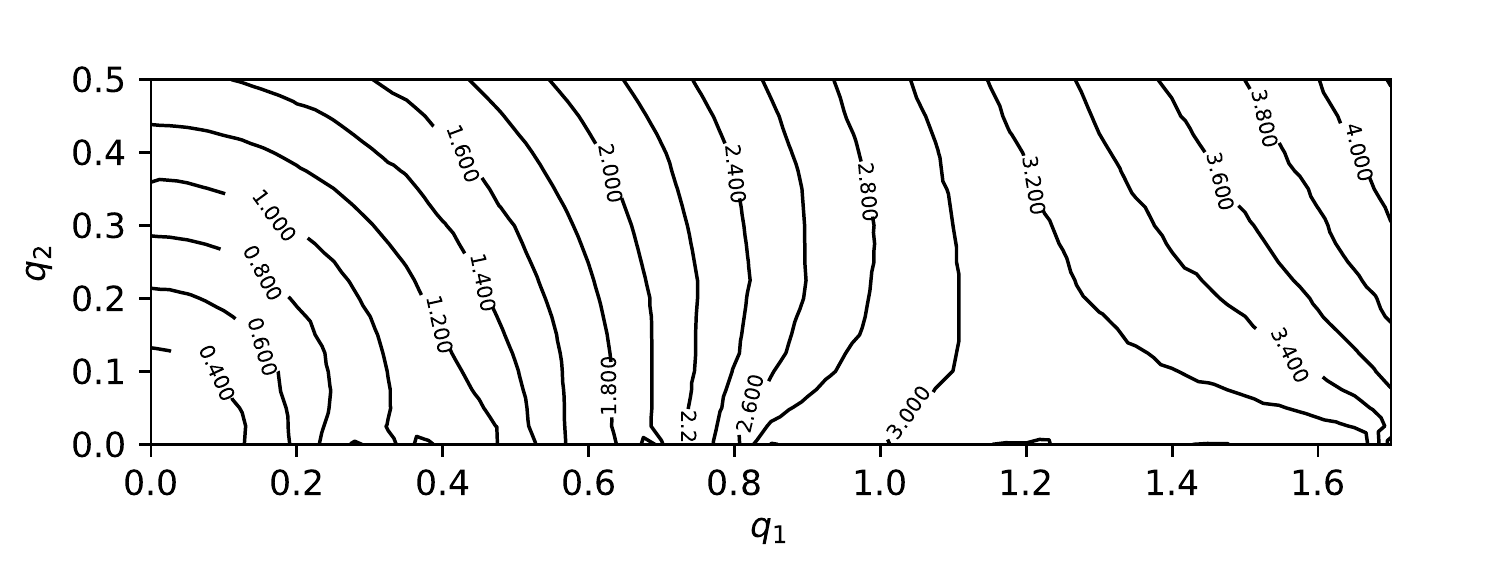}
	\caption{Contour plot of $r(q)$ with $g(x,t)= \sin (2 \pi (x_1+t)) + \sin (2 \pi (x_1+3t)) +3$ computed with $M = 256$. The two visible pinning intervals along the $q_1$ axis where the velocity is pinned to $1$ and $3$ match those in Figure~\ref{fig:rq}.}
	\label{fig:rg_double}
\end{figure}

\begin{figure}
	\centering
	\includegraphics[clip, trim=0.5cm 0cm 0.5cm 0cm, width=\textwidth]{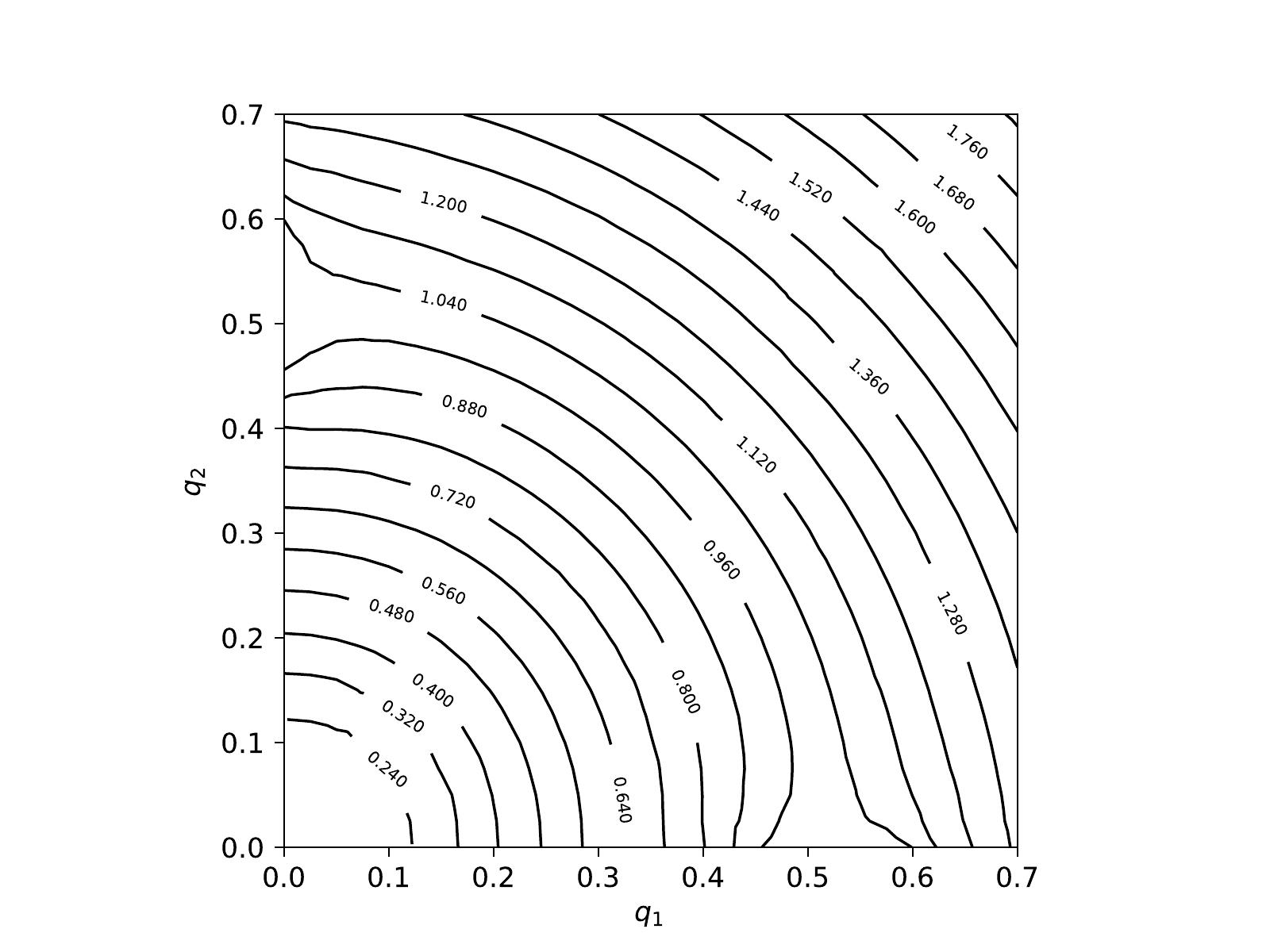}
	\caption{Contour plot of $r(q)$ with $g(x,t)= \frac 12 \cos(2 \pi t) (\sin (2 \pi x_1) + \sin (2 \pi x_2)) +2$ with $M = 256$. Pinning intervals appear in the directions of $(1,0)$, $(-1,0)$, $(0,1)$ and $(0,-1)$ (the plot is symmetric with respect to the rotation by $\frac \pi2$).}
	\label{fig:rg4}
\end{figure}

\begin{figure}
	\centering
	\subfloat{
		\includegraphics[clip,trim=0cm 0cm 0cm 1.5cm,width = 0.8 \textwidth]{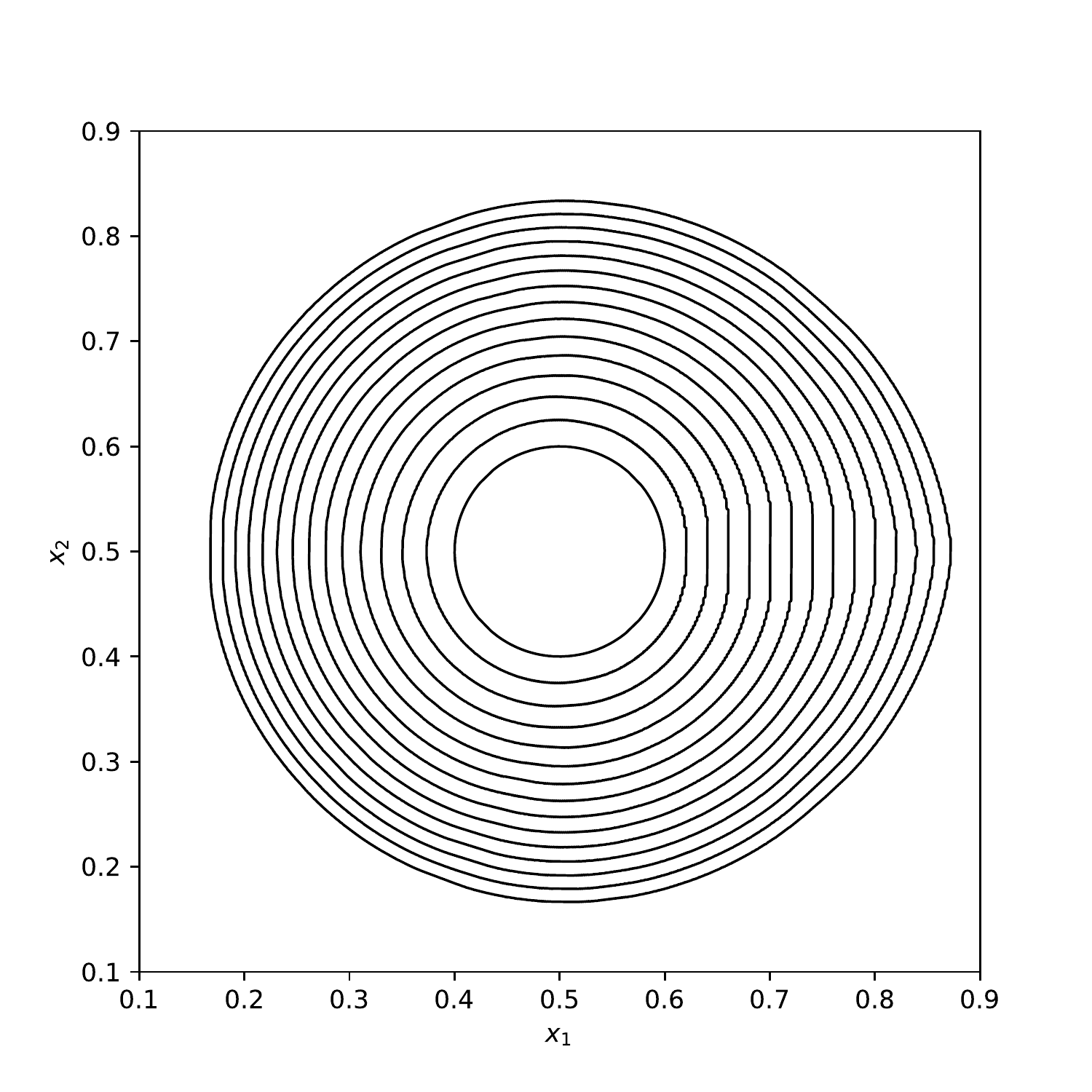}
	}\\
	\subfloat{%
		\includegraphics[clip,trim=0.cm 3cm 0.5cm 4cm,width=0.8\textwidth]{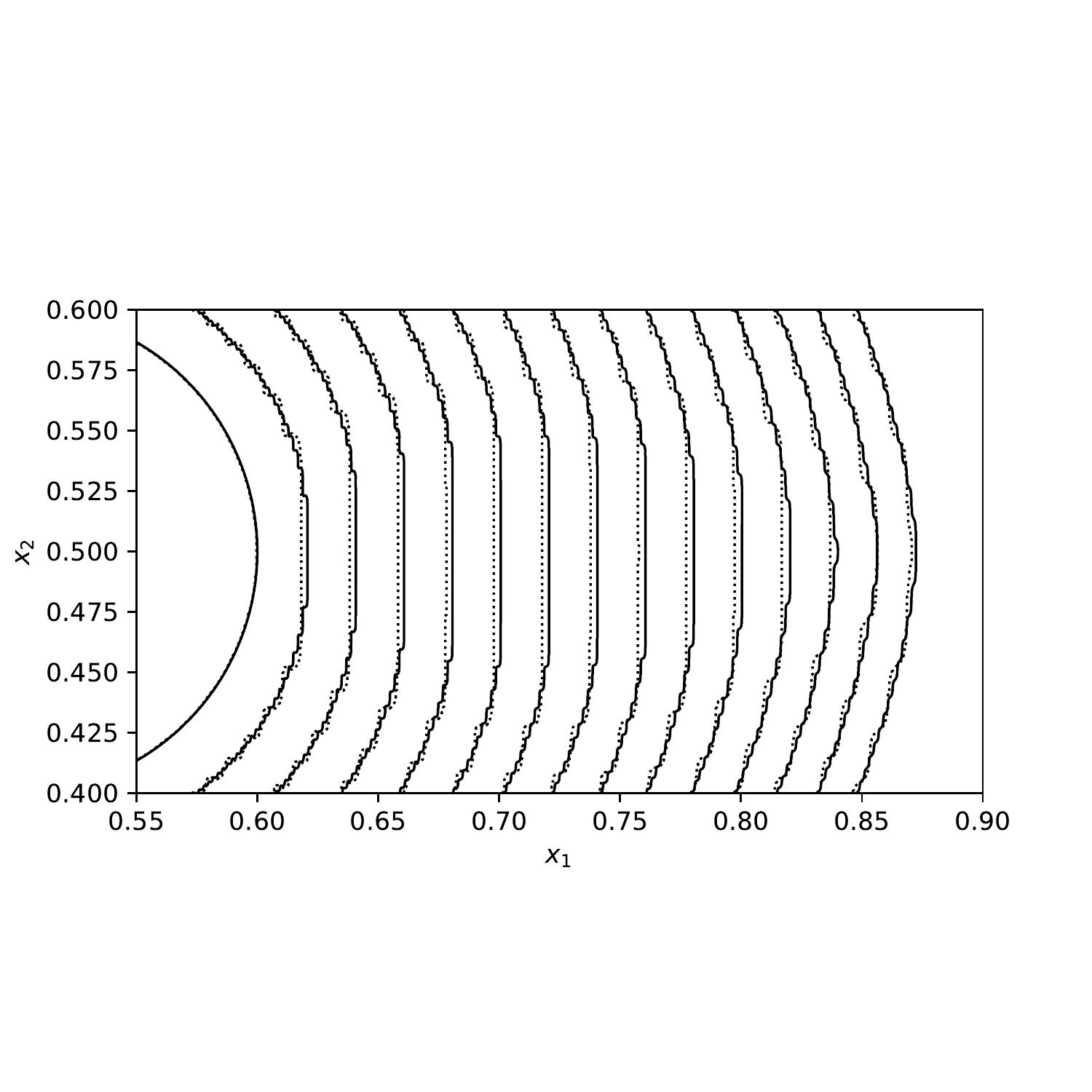}
	}
	\caption{(Top) The free boundary of the numerical solution of the Hele-Shaw problem with a given source $f=1500 \max\{0.1 - |x - (\frac 12, \frac 12)|,0\}$ and a function $g(x,t)=\sin(2 \pi (x_1 - t)) + 1.05$ with initial data $\Omega_0 = \set{x: |x - (\frac 12, \frac 12)| < 0.1}$. The free boundary is plotted at times $t = 0.02 m$, $m \in N$. A facet seems to appear in direction $(1, 0)$. It reaches its maximum length at $t\approx 0.12 $. Solid line is the solution with $M = 8192$, $\e = \frac 1{512}$, while the dotted line is with $M = 2048$, $\e = \frac 1{128}$. We used 1 V-cycle. (Bottom) Detail of the region with facets.}
	\label{fig:facet}
\end{figure}

\clearpage
\begin{table}
\begin{center}
\begin{tabular}{c|c|c|c|c|c}
\multirow{2}{*}{$M$} & \multicolumn{5}{c}{$\e^{-1}$} \\
\cline{2-6}
& $8$ & $16$ & $32$ & $64$ & $128$ \\
\hline
$64$ & $8.7\times 10^{-2}$ & $6.3\times 10^{-2}$ & $5.2\times 10^{-1}$ & $1.0\times 10^{0}$ & $1.5\times 10^{0}$ \\
$128$ & $5.9\times 10^{-2}$ & $5.3\times 10^{-2}$ & $5.9\times 10^{-2}$ & $3.9\times 10^{-1}$ & $1.0\times 10^{0}$ \\
$256$ & $6.5\times 10^{-2}$ & $5.8\times 10^{-2}$ & $2.1\times 10^{-2}$ & $5.0\times 10^{-2}$ & $4.0\times 10^{-1}$ \\
$512$ & $7.9\times 10^{-2}$ & $5.7\times 10^{-2}$ & $2.0\times 10^{-2}$ & $1.1\times 10^{-2}$ & $5.9\times 10^{-2}$ \\
\end{tabular}
\end{center}
\caption{The maximum of the error of the numerical estimate of $r(q)$ for $g(x,t) = \sin(2\pi(x_1 + t)) + 2$ as compared to the estimate of $r(q)$ using the ODE method in one dimension described in Section~\ref{sec:one-dimension} for $q = (q_1, 0)$, with a small sample of values $q_1$ away from the ends of the pinning interval $[\frac 13, 1]$ for various values of parameters $M$ and $\e$. Our method appears to be first order accurate in $M$ if $\e$ is chosen appropriately.}
\label{tab:accuracy}
\end{table}

\section*{Acknowledgments}
The first author acknowledges the support of Ministry of Education, Culture, Sports, Science, and Technology of Japan (MEXT) through the Program for the Development of Global Human Resources in Mathematical and Physical Sciences at Kanazawa University. The first author would also like to show her sincere gratitude to Professor Seiro Omata for his support and guidance.
The second author was partially supported by JSPS KAKENHI Grant No. 18K13440 (Wakate).


\end{document}